\documentclass{amsart} 
\usepackage{graphicx}
\usepackage{hyperref}
\usepackage{amsmath, amsthm, amssymb, amscd}

\usepackage{mathrsfs}
\usepackage{mathtools}
\usepackage{enumerate}
\usepackage{hyperref}
\usepackage{verbatim}
\usepackage{centernot}
\usepackage{subcaption}
\usepackage{cite}
\usepackage{color}
\usepackage{esint}
\usepackage{tikz-cd}
\usetikzlibrary{shapes.geometric}
\usepackage[shortlabels]{enumitem}
 
\usepackage{bm}
\usepackage{bbm, dsfont}
\definecolor{mycolor}{rgb}{0.0, 0.75, 1.0}

\makeatletter
\newcommand{\labitem}[2]{%
\def\@itemlabel{\textbf{#1}}
\item
\def\@currentlabel{#1}\label{#2}}
\makeatother

\title[Reflected diffusion on inner uniform domains]{Sub-Gaussian heat kernel estimates for reflected diffusion on inner uniform domains}
\author{Riku Anttila}
\address[Riku Anttila]{Department of Mathematics and Statistics, University of Jyväskylä, P.O. Box 35, FI-40014 Jyväskylä, Finland}
\email{riku.t.anttila@jyu.fi}


\date{\today}

\usepackage[shortlabels]{enumitem}
 
\newtheorem{theorem}[equation]{Theorem}
\newtheorem{lemma}[equation]{Lemma}

\newtheorem{proposition}[equation]{Proposition}
\newtheorem{corollary}[equation]{Corollary}

\newtheorem{question}[equation]{Question}
\newtheorem*{theorem*}{Theorem}
\newtheorem*{question*}{Question}
\newtheorem*{mainQ*}{Main question}

\numberwithin{equation}{section}

\theoremstyle{definition}
\newtheorem{definition}[equation]{Definition}

\theoremstyle{remark}
\newtheorem{remark}[equation]{Remark}
    \newcommand*{\N}{\mathbb{N}}
    
    \newcommand*{\R}{\mathbb{R}}
    
        \DeclarePairedDelimiter\Span{\langle}{\rangle}
        
        \DeclareMathOperator{\diam}{diam}
        
        \DeclareMathOperator{\dist}{dist}

        \DeclareMathOperator{\len}{len}

        \DeclareMathOperator{\loc}{loc}

        \DeclareMathOperator{\supp}{supp}

        \DeclarePairedDelimiter\abs{\lvert}{\rvert}
        \DeclarePairedDelimiter\norm{\lVert}{\rVert}
        
        \DeclarePairedDelimiter\floor{\lfloor}{\rfloor}

        \def\vint_#1{\mathchoice%
          {\mathop{\kern 0.2em\vrule width 0.6em height 0.69678ex depth -0.58065ex
                  \kern -0.8em \intop}\nolimits_{\kern -0.4em#1}}%
          {\mathop{\kern 0.1em\vrule width 0.5em height 0.69678ex depth -0.60387ex
                  \kern -0.6em \intop}\nolimits_{#1}}%
          {\mathop{\kern 0.1em\vrule width 0.5em height 0.69678ex
              depth -0.60387ex
                  \kern -0.6em \intop}\nolimits_{#1}}%
          {\mathop{\kern 0.1em\vrule width 0.5em height 0.69678ex depth -0.60387ex
                  \kern -0.6em \intop}\nolimits_{#1}}}
\def\vintslides_#1{\mathchoice%
          {\mathop{\kern 0.1em\vrule width 0.5em height 0.697ex depth -0.581ex
                  \kern -0.6em \intop}\nolimits_{\kern -0.4em#1}}%
          {\mathop{\kern 0.1em\vrule width 0.3em height 0.697ex depth -0.604ex
                  \kern -0.4em \intop}\nolimits_{#1}}%
          {\mathop{\kern 0.1em\vrule width 0.3em height 0.697ex depth -0.604ex
                  \kern -0.4em \intop}\nolimits_{#1}}%
          {\mathop{\kern 0.1em\vrule width 0.3em height 0.697ex depth -0.604ex
                  \kern -0.4em \intop}\nolimits_{#1}}}

\newcommand{\aveint}[2]{\mathchoice%
          {\mathop{\kern 0.2em\vrule width 0.6em height 0.69678ex depth -0.58065ex
                  \kern -0.8em \intop}\nolimits_{\kern -0.45em#1}^{#2}}%
          {\mathop{\kern 0.1em\vrule width 0.5em height 0.69678ex depth -0.60387ex
                  \kern -0.6em \intop}\nolimits_{#1}^{#2}}%
          {\mathop{\kern 0.1em\vrule width 0.5em height 0.69678ex depth -0.60387ex
                  \kern -0.6em \intop}\nolimits_{#1}^{#2}}%
          {\mathop{\kern 0.1em\vrule width 0.5em height 0.69678ex depth -0.60387ex
                  \kern -0.6em \intop}\nolimits_{#1}^{#2}}}
\usepackage{systeme}
\makeatletter
\let\c@equation\c@figure
\makeatother

\begin{document}

\begin{abstract}
We prove that sub-Gaussian heat kernel estimates are inherited from a diffusion process on the ambient space to the reflected diffusion process on a subset which is an inner uniform domain.
\end{abstract}

\maketitle
\section{Introduction}

The objective of this work is to positively answer a question posed by Murugan concerning heat kernel estimates of reflected diffusion \cite[Section 6.3]{murugan2024heat}. Following \cite{murugan2024heat}, our methods and results are formulated in terms of abstract Dirichlet forms. This approach relies on the well-known correspondence between regular Dirichlet forms and a certain class of Markov processes, established in the celebrated theorem of Fukushima \cite[Theorems 7.2.1-7.2.2]{FOT}.

The setting of the present work is as follows; see the discussion below and also Section \ref{sec:Preli} for the precise terminology and the relevant assumptions.
We let $(X,d,\mu)$ be a suitably regular metric measure space and consider a strongly local regular Dirichlet form $(\mathcal{E},\mathcal{F})$ on $L^2(X,\mu)$.
We also let $\{Y_t\}_{t > 0}$ be the associated diffusion process and assume it admits transition probability densities $\{p_t\}_{t > 0}$, meaning
\[
    \mathbb{P}\left(Y_t \in A \, \big |\, Y_0 = x \right) = \int_{A} p_t(x,y) \, d\mu(y)
\]
for all $x \in X \setminus E,\, t > 0$ and a Borel set $A \subseteq X$ where $E \subseteq X$ is a properly exceptional subset.

Conceptually, the reflected diffusion on an open subset $\Omega \subseteq X$, or more precisely on a certain completion of $\Omega$, can be understood as a stochastic process $\{\widetilde{Y}_t\}_{t > 0}$ that behaves like $\{Y_t\}_{t > 0}$ inside $\Omega$, but is pushed back into the domain upon hitting the boundary.
A classical example is the \emph{normally reflected Brownian motion} on a smooth domain $\Omega \subseteq \R^n$ studied in the SDE literature; see \cite{NRBM04} and references therein for further discussion.
See also \cite[Introduction and Section 2]{murugan2024heat}, \cite[Chapter 6]{chen2012symmetric} and references therein for a comprehensive overview on reflected diffusion.

The interest of the work is in understanding, when do nice properties of the diffusion $\{Y_t\}_{t > 0}$ on the ambient space $X$ inherit to the reflected diffusion $\{\widetilde{Y}_t\}_{t > 0}$ on $\Omega$.
In this direction, Murugan posed the following question.


\begin{question}[Section 6.3 \cite{murugan2024heat}]
\label{question}
    Assume that $\Omega \subseteq X$ is an inner uniform domain and that $\{p_t\}_{t > 0}$ satisfies the heat kernel estimates \hyperref[eq:HKMain]{$\textup{HKE}(\beta)$} for $\beta \geq 2$.
    Then, is it always true that the reflected diffusion $\{\widetilde{Y}_t\}_{t > 0}$ on $\Omega$ also admits transition probability densities $\{\widetilde{p}_t\}_{t > 0}$ satisfying heat kernel estimates \hyperref[eq:HKMain]{$\textup{HKE}(\beta)$}?
\end{question}

Question \ref{question} arises naturally since the analogous results are known for the Gaussian heat kernel estimates, namely when $\{p_t\}_{t > 0}$ satisfies \hyperref[eq:HKMain]{$\textup{HKE}(\beta)$} for $\beta = 2$. This was established by Gyrya and Saloff-Coste \cite[Theorem 3.10]{SaloffCoste2011NeumannAD}. In the strict sub-Gaussian case $\beta > 2$, Murugan positively answered Question \ref{question} when $\Omega$ satisfies a stronger assumption that $\Omega \subseteq X$ is a uniform domain \cite[Theorem 2.8]{murugan2024heat}.

The methods of Gyrya and Saloff-Coste do not immediately extend to the sub-Gaussian case. The main difficulty arises from the \emph{cutoff Sobolev inequality}, which is an energy inequality that plays a central role in the characterization of sub-Gaussian heat kernel estimates \cite{barlow2004stability,barlow2006stability,GrigorCapacity15}, but is irrelevant for the Gaussian case.
Establishing the cutoff Sobolev inequality directly for reflected diffusion on an inner uniform domain, or even on a uniform domain, appears to be a rather difficult task. In \cite{murugan2024heat}, Murugan circumvented this issue by showing that uniform domains are (Sobolev) extension domains in an analogous sense to a famous result of Jones \cite{Jones81}. This can be used to reduce many problems on the domain $\Omega$ to the ambient space $X$ where the necessary energy inequalities, including the cutoff Sobolev inequality, are already available.
However, such an approach does not work in general for inner uniform domains. For instance, it is well known in the literature of Sobolev spaces that a slit disk is not a Sobolev extension domain, but it is, nevertheless, an inner uniform domain.

The main result of this work is a positive answer to Question \ref{question}.
We formulate it in the language of Dirichlet forms, and see Theorem \ref{thm:MainMain} for a more general version.

\begin{theorem}[Theorem \ref{thm:MainMain}]\label{thm:Main}
    Let $(X,d,\mu)$ be a metric measure space where $(X,d)$ is complete and geodesic, and $\mu$ is doubling measure on $(X,d)$. Let $(\mathcal{E},\mathcal{F})$ be a strongly local regular Dirichlet form on $L^2(X,\mu)$ satisfying \hyperref[eq:HKMain]{$\textup{HKE}(\beta)$} for $\beta \geq 2$, $\Omega \subseteq X$ be an inner uniform domain and $(\widetilde{\Omega},\rho)$ be the completion of $(\Omega,\rho)$ where $\rho$ is the path metric of $\Omega$.
    Then the bilinear form $(\mathcal{E}^\Omega,\mathcal{F}(\Omega))$ in Definition \ref{def:LocDir} is a strongly local regular Dirichlet form on $L^2(\widetilde{\Omega},\mu)$ satisfying \hyperref[eq:HKMain]{$\textup{HKE}(\beta)$}.
\end{theorem}

Theorem \ref{thm:Main} relates to Question \ref{question} by the fact that, assuming $(\mathcal{E}^\Omega,\mathcal{F}(\Omega))$ is a strongly local regular Dirichlet form on $L^2(\widetilde{\Omega},\mu)$, then by definition, $(\mathcal{E}^\Omega,\mathcal{F}(\Omega))$ is the Dirichlet form corresponding to the reflected diffusion process on $\widetilde{\Omega}$.
This definition is well-defined due to the aforementioned theorem of Fukushima.

As we already indicated, the main challenge in the proof of Theorem \ref{thm:Main} is the treatment of the cutoff Sobolev inequality.
To this end, the key ingredient of our approach is to use a characterization of the sub-Gaussian heat kernel estimates, recently obtained by the author, which does not involve the cutoff Sobolev inequality \cite{anttila2025approach}. Instead, it has been replaced with a new condition, the \emph{cutoff energy condition}, which we prove to be inherited to inner uniform domains using a result of Väisälä from the late 90s \cite{vaisala1998}.
The cutoff energy condition also plays a central role in the proof of the regularity of $(\mathcal{E}^\Omega,\mathcal{F}(\Omega))$.

\subsection*{Organization of the paper}
In Section \ref{sec:Preli}, we recall the necessary terminology for our proofs and results.

We prove some helpful lemmas in Section \ref{sec:aux}, which for the most part are already known in the literature.

In Section \ref{sec:MainThm}, we prove the main result of the work, Theorem \ref{thm:Main}/Theorem \ref{thm:MainMain}.

\section*{Acknowledgments}
This work started from a problem suggested by Mathav Murugan, and I am deeply grateful for his encouragement and for many valuable discussions.
I also thank Naotaka Kajino for helpful discussions and suggestions.
This work is supported by the Finnish Ministry of Education and Culture’s Pilot for Doctoral Programmes (Pilot project Mathematics of Sensing, Imaging and Modelling).

\section{Preliminary}\label{sec:Preli}

We begin by fixing the framework and the terminology.

\subsection{Metric spaces}
We recall some standard terminology of metric spaces and measures; see \cite{CourseMG,Heinonen} for further background.

Let $(X,d)$ be a complete metric space and denote its \emph{open balls}
\[
B(x,r) := \{y\in X: d(x,y) < r\}.
\]
The \emph{distance} between a point $x \in X$ and a non-empty subset $A \subseteq X$ is denoted
\[
    \dist(x,A) := \inf \{ d(x,a) : a \in A \},
\]
and the \emph{diameter} of a non-empty subset $A \subseteq X$ is
\[
    \diam(A) := \sup_{x,y \in A} d(x,y)
\]
The set of continuous functions $X \to \R$ is denoted $C(X)$, and its subset consisting of compactly supported continuous functions is $C_c(X)$. Given Borel subsets $E \subseteq F$ of $X$, we say that $f : X \to \R$ is \emph{cutoff function} for $E \subseteq F$ if $f \in C_c(X)$, $f|_E = 1$, $f|_{X \setminus F} = 0$ and $0 \leq f(x) \leq 1$ for all $x \in X$.

We say that a Borel measure $\mu$ on $(X,d)$ is \emph{doubling} if there is $D\geq 1$ such that
\begin{equation}\label{eq:mudoubling}
    0 < \mu(B(x,2r)) \leq D\mu(B(x,r)) < \infty
\end{equation}
for all $x \in X$ and $r > 0$.
If a doubling measure exists, then it follows from a direct volume argument that the metric space $(X,d)$ satisfies the \emph{metric doubling property}. This means that there exists a constant $N = N(D) \in \N$ such that for every $x\in X$ and $r>0$ there are $x_1,\ldots , x_N\in X$,
\begin{equation}\label{eq:MD}
    B(x,2r) \subseteq \bigcup_{i = 1}^N B(x_i,r).
\end{equation}
In particular, since we assumed $(X,d)$ to be complete, it is therefore also \emph{proper}, meaning every bounded closed set is compact. Furthermore, the metric doubling property clearly implies separability.

We also recall some terminology of rectifiable curves.
First, a \emph{curve} in a metric space $(X,d)$ is a continuous function $\gamma : [0,1] \to X$.
We often regard $\gamma$ as a subset of $X$ with the obvious interpretations. For instance, we write $z\in \gamma$ to indicate $z = \gamma(t)$ for some $t\in [0,1]$, or $\gamma \subseteq \Omega$ if $\gamma([0,1]) \subseteq \Omega$.
The \emph{length} of $\gamma$ is the value
\begin{equation*}
    \len(\gamma) := \sup  \left\{ \sum_{i = 0}^{k-1} d(\gamma(t_i),\gamma(t_{i+1})) : \text{ $\{ t_i \}_{i = 0}^k \subseteq [0,1]$ is an increasing sequence} \right\}.
\end{equation*}
We say that $\gamma$ is a \emph{rectifiable curve} it is has finite length.
Lastly, we say that a metric space $(X,d)$ is \emph{geodesic} if for every pair of points $x,y \in X$ there is a rectifiable curve $\gamma$ such that $x = \gamma(0),\, y = \gamma(1)$ and $d(x,y) = \len(\gamma)$.





\subsection{Inner uniform domains}
We recall the concepts related to inner uniform domains. See the work of Väisälä \cite{vaisala1998} for an excellent exposition on the topic. See also \cite{Lierl2014,SaloffCoste2011NeumannAD,Lierl2022,murugan2024heat,ReflDiffJumps2022,kajino2023heat,lierl2014dirichlet} for studies related to inner uniform domains in the Dirichlet form setting.

Let $(X,d)$ be a geodesic metric space and $\Omega \subseteq X$ be a non-empty open subset.
We define the \emph{path metric} of $\Omega$ by
\[
\rho(x,y) := \inf_{\gamma} \len(\gamma)
\]
where the infimum is taken over all rectifiable curves $\gamma$ in $X$ with $x = \gamma(0),\, y = \gamma(1)$ and $\gamma \subseteq \Omega$.
Here we understand $\inf \emptyset = \infty$.


\begin{definition}\label{def:IUD}
    Let $\Omega \subseteq X$ be a non-empty open subset and $\rho$ be its path metric. We say that $\Omega \subseteq X$ is an \emph{inner uniform domain} if
    there are constants $C_0,c_0 > 0$ such that
    for every pair $x,y \in \Omega$ there is a rectifiable curve $\gamma : [0,1] \to X$ with $\gamma \subseteq \Omega$, $x = \gamma(0),\, y = \gamma(1)$ and it satisfies the following two conditions.
    \begin{enumerate}
        \vspace{2pt}
        \item $\len(\gamma) \leq C_0 \cdot \rho(x,y)$.
        \vspace{2pt}
        \item For every $z \in \gamma$,
        \[
            \dist(z,X\setminus\Omega) \geq c_0 \frac{\rho(x,z) \cdot \rho(y,z)} {\rho(x,y)}.
        \]
    \end{enumerate}
    Here both $\len(\gamma)$ and $\dist(z,X \setminus \Omega)$ are computed using the ambient metric $d$.
\end{definition}

\begin{remark}
    We note that the length of a curve $\gamma \subseteq \Omega$ computed using the path metric $\rho$ coincides with the length computed with the ambient metric $d$; we refer to \cite[Proposition 2.3.12]{CourseMG} for details. Also note that, if $(\widetilde{\Omega},\rho)$ denotes the completion of $(\Omega,\rho)$, then it follows from the fact that $(X,d)$ is geodesic that
    \[
        \dist(z,X\setminus\Omega) =  \dist(z,\widetilde{\Omega}\setminus \Omega) :=  \inf \{ \rho(z,a) : a \in \widetilde{\Omega} \setminus \Omega \}
    \]
\end{remark}



\subsection{Dirichlet forms}
Next, we fix the terminology of Dirichlet forms and refer to the standard references \cite{FOT,chen2012symmetric} for further details.
We shall follow the convention that, given any two variable object $L : V \times V \to Z$, we denote $L(v) := L(v,v)$.

We note that some suitable assumptions on the ambient space $(X,d,\mu)$ are required for the following definition to be well-defined. In this work, we assume $(X,d)$ to be complete and geodesic, and that $\mu$ is a doubling Radon measure on $(X,d)$, which are sufficient to this end.

\begin{definition}\label{def:Dir}
Let $(X,d,\mu)$ be a metric measure space.
We say that $(\mathcal{E},\mathcal{F})$ is a \emph{Dirichlet form} on $L^2(X,\mu)$ if the following two conditions hold.
    \begin{enumerate}
        \vspace{3pt}
        \item $\mathcal{E}: \mathcal{F} \times \mathcal{F} \to \R$ is a symmetric non-negative definite bilinear form such that $\mathcal{F} \subseteq L^2(X,\mu)$ is a dense linear subspace, and $\mathcal{F}$ equipped with the inner product $\mathcal{E}_1(f,g) := \mathcal{E}(f,g) + \int_X f \cdot g \, d\mu$ is a Hilbert space.
        \vspace{3pt}
        \item For all $f \in \mathcal{F}$ we have $f^+ \land 1 \in \mathcal{F}$ and $\mathcal{E}(f^+ \land 1,f^+ \land 1) \leq \mathcal{E}(f,f)$. This condition is called the \emph{Markov property}.
    \end{enumerate}
    We consider two additional conditions.
    \begin{enumerate}
        \vspace{3pt}
        \item[(3)] We say that a Dirichlet form $(\mathcal{E},\mathcal{F})$ on $L^2(X,\mu)$ is \emph{regular} if the subspace $\mathcal{F} \cap C_c(X)$ is dense in both the inner product space $(\mathcal{F},\mathcal{E}_1)$ and the normed space $(C_c(X),\norm{\cdot}_{L^\infty})$.
        \vspace{3pt}
        \item[(4)] We say that a Dirichlet form $(\mathcal{E},\mathcal{F})$ on $L^2(X,\mu)$ is \emph{strongly local} if the following implication always holds. Whenever $f,g \in \mathcal{F}$ such that their supports $\supp_\mu[f], \supp_\mu[g] \subseteq X$ are compact and there is $a \in \R$ such that $\supp_{\mu}[f] \cap \supp_{\mu}[g - a \mathds{1}_X] = \emptyset$, we have $\mathcal{E}(f,g) = 0$.
        Here $\mathds{1}_X$ denotes the constant function $x \mapsto 1$, and $\supp_{\mu}[f]$ is the smallest closed set $F \subseteq X$ with $\int_{X\setminus F} f d\mu = 0$.
    \end{enumerate}
    Given a strongly local regular Dirichlet form $(\mathcal{E},\mathcal{F})$ on $L^2(X,\mu)$, every $f \in \mathcal{F}$ is assigned the associated \emph{energy measure} $\Gamma\Span{f,f} = \Gamma\Span{f}$ as follows; see \cite[Chapter 3]{FOT} for details.
    If $f \in \mathcal{F}\cap L^\infty(X,\mu)$, then $\Gamma\Span{f}$ is the unique non-negative Radon measure on $(X,d)$ satisfying
    \[
        \int_X \varphi \, d\Gamma\Span{f} = \mathcal{E}(f,f\varphi) - \frac{1}{2}\mathcal{E}(f^2,\varphi) \quad \text{for all } \varphi \in \mathcal{F} \cap C_c(X).
    \]
    For a general $f \in \mathcal{F}$, we now define $\Gamma\Span{f}(A) = \lim_{k \to \infty} \Gamma\Span{(f \lor -k) \land k}(A)$. We also define the two variable energy measures as the signed Radon measures
    \[
    \Gamma\Span{f,g} := 1/4(\Gamma\Span{f+g}-\Gamma\Span{f-g}).
    \]
\end{definition}

\begin{remark}
    Following the previous definitions, we note that $\Gamma\Span{f,g}(X) = \mathcal{E}(f,g)$ for all $f,g \in \mathcal{F}$ by \cite[Lemma 3.2.3]{FOT}. Moreover, $\mathcal{F} \cap L^\infty(X,\mu)$ is an algebra according to \cite[Theorem 1.4.2-(ii)]{FOT}, namely $f \cdot g \in \mathcal{F}$ for all $f,g \in \mathcal{F} \cap L^\infty(X,\mu)$.
\end{remark}

\begin{definition}\label{def:LocDir}
    Given a strongly local regular Dirichlet form $(\mathcal{E},\mathcal{F})$ on $L^2(X,\mu)$ and a non-empty open subset $\Omega \subseteq X$, we define the associated \emph{local Dirichlet space}
    \begin{equation*}\label{eq:LocalDSpace}
        \mathcal{F}_{\loc}(\Omega) := \left\{\, f :
        \begin{array}{lcr}
        \text{$f$ is a $\mu$-equivalence class of Borel functions $\Omega \to \R$} \\
        \text{such that $\mathds{1}_V f = \mathds{1}_V f^{\#}$ $\mu$-a.e. for some $f^\# \in \mathcal{F}$} \\
        \text{for each relatively compact subset $V \Subset \Omega$}
\end{array}
        \right\}.
    \end{equation*}
    Each pair $f,g \in \mathcal{F}_{\loc}(\Omega)$ is assigned the \emph{energy measure}, which is the unique signed Radon measure given by $\Gamma^\Omega \Span{f,g}(A) := \Gamma\Span{f^{\#},g^{\#}}(A)$ for all relatively compact subsets $A \Subset V$ where $f^{\#}, g^\# \in \mathcal{F}$ and $V \Subset \Omega$ are as in the previous display.
    Lastly, we define the bilinear form $(\mathcal{E}^\Omega,\mathcal{F}(\Omega))$ where $\mathcal{E}^\Omega : \mathcal{F}(\Omega) \times \mathcal{F}(\Omega) \to \R$ is given by
    \[
    \mathcal{E}^\Omega(f,g) := \int_\Omega \, d\Gamma^\Omega\Span{f,g},
    \]
    and $\mathcal{F}(\Omega)$ is the function space
    \[
    \mathcal{F}(\Omega) := \left\{f\in \mathcal{F}_{\loc}(\Omega) : \mathcal{E}_1^\Omega(f,f) <\infty\right\}
    \]
    where $\mathcal{E}_1^\Omega$ is the inner product
    \[
   \mathcal{E}_1^\Omega(f,g) := \int_\Omega f \cdot g \,d\mu + \mathcal{E}^\Omega(f,g).
    \]
\end{definition}

\begin{remark}
    The pair $(\mathcal{E}^\Omega,\mathcal{F}(\Omega))$ is not necessarily a strongly local regular Dirichlet form on $L^2(\Omega,\mu)$.
    Indeed, consider the following example. Let $(\mathcal{E},\mathcal{F})$ be the usual Dirichlet energy in $\R^n$, namely $\mathcal{F} := W^{1,2}(\R^n)$ and
    \[
        \mathcal{E}(f,g) := \int_{\R^n} \Span{\nabla f, \nabla g} \,dx.
    \]
    If $\Omega := B(0,1) \subseteq \R^n$ is the unit ball, then Dirichlet form $(\mathcal{E}^\Omega,\mathcal{F}(\Omega))$ on $L^2(\Omega,dx)$ is not regular.
    To see this, simply take the constant function $\mathds{1}_\Omega$, or more generally any function $f \in W^{1,2}(\R^n)$ such that $f|_{\Omega} \notin W^{1,2}_0(\Omega)$.
    Such functions cannot be approximated by $\{f_n\}_{n = 1}^\infty \subseteq C_c(\Omega)$ in the Sobolev norm.
    This issue can be sometimes resolved by taking the completion of $\Omega$. 
    Indeed, in the case discussed here, $(\mathcal{E}^\Omega,\mathcal{F}(\Omega))$ is a strongly local regular Dirichlet form on $L^2(\overline{\Omega},dx)$.
\end{remark}

\subsection{Heat kernel estimates}
We recall the terminology related to heat kernel estimates; see \cite{grigor2012two,grigor2014heat,saloff2002aspects} and therein references for further literature.

\begin{definition}  We say that $\Psi : (0,\infty) \to (0,\infty)$ is a \emph{scale function} if it is an increasing homeomorphism satisfying the following doubling type property for some $\beta_U,\beta_L > 1$. There is $C \geq 1$ such that for all $0 < r \leq R$,\begin{equation}\label{eq:PsiRad}C^{-1}\left(\frac{R}{r}\right)^{\beta_L} \leq \frac{\Psi(R)}{\Psi(r)} \leq C\left(\frac{R}{r}\right)^{\beta_U}.\end{equation}Given such $\Psi$, we associate it a function $\Phi$ given by\begin{equation}\label{eq:PhiRad}\Phi(s) := \sup_{r > 0} \left(\frac{s}{r} - \frac{1}{\Psi(r)}\right).\end{equation}\end{definition}

\begin{definition}\label{def:heatkernel}
    Let $(\mathcal{E},\mathcal{F})$ be a strongly local regular Dirichlet form on $L^2(X,\mu)$ and $\{P_t\}_{t > 0}$ be the associated Markov semigroup; see \cite[Section 1.4]{FOT}.
    A family of Borel measurable function $\{p_t\}_{t > 0}$, $p_t : X \times X \to [0,\infty]$, is a \emph{heat kernel} of $(\mathcal{E},\mathcal{F})$ if for every $t > 0$ the function $p_t$ is an integral kernel of $P_t$, meaning for all $t > 0$ and $f \in L^2(X,\mu)$,
    \[
    P_t(f)(x) = \int_X p_t(x,y)f(y)\, d\mu \quad \text{for $\mu$-almost every $x \in X$.}
    \]
    Given a scale function $\Psi$, we say that $(\mathcal{E},\mathcal{F})$ satisfies the \emph{heat kernel estimates} \ref{eq:HKMain} if there is a heat kernel $\{p_t\}_{t > 0}$ of $(\mathcal{E},\mathcal{F})$ and constants $C,C_1,C_2,c,\kappa > 0$ such that for all $t > 0$,
    \begin{align}\label{eq:HKMain}
    \tag{$\textup{HKE}(\Psi)$}
    p_t(x,y) & \leq \frac{C}{\mu(B(x,\Psi^{-1}(t))}\exp\left( -C_1t\Phi \left(C_2\frac{d(x,y)}{t}\right) \right) \text{ for $\mu$-a.e. $x,y \in X$}, \\
    p_t(x,y) & \geq \frac{c}{\mu(B(x,\Psi^{-1}(t)))} \text{ for $\mu$-a.e. $x,y \in X$ with $d(x,y) \leq \kappa \Psi^{-1}(t)$}. \nonumber
\end{align}
If the scale function is given by $\Psi(r) = r^\beta$ for some $\beta \geq 2$, we say that $(\mathcal{E},\mathcal{F})$ satisfies the \emph{heat kernel estimates} \hyperref[eq:HKMain]{$\textup{HKE}(\beta)$}.
\end{definition}

\begin{definition}\label{def:PI(Dir)}
    Given a scale function $\Psi$, we say that a strongly local regular Dirichlet form $(\mathcal{E},\mathcal{F})$ on $L^2(X,\mu)$ satisfies the \emph{Poincar\'e inequality} \ref{eq:PI} if there are constants $C,\sigma \geq 1$ satisfying the following condition.
    For all $x \in X$, $r > 0$ and $f \in \mathcal{F}$,
    \begin{equation}\label{eq:PI}
        \tag{$\textup{PI}(\Psi)$}
        \int_{B(x,r)} (f - f_{B(x,r)})^2 \, d\mu \leq C \Psi(r) \int_{B(x,\sigma r)} \, d\Gamma\Span{f}.
    \end{equation}
\end{definition}

\begin{definition}\label{def:CE}
     Given a scale function $\Psi$ and $\delta > 0$, we say that the strongly local regular Dirichlet form $(\mathcal{E},\mathcal{F})$ on $L^2(X,\mu)$ satisfies the \emph{cutoff energy condition} \ref{eq:CE} if there is a constant $C \geq 1$ satisfying the following. For all $x \in X$ and $R \in (0,\diam(X))$ there is a cutoff function $\xi \in \mathcal{F}$ for $B(x,r) \subseteq B(x,2r)$ such that
    \begin{equation}\label{eq:CE}
        \tag{$\textup{CE}_{\delta}(\Psi)$}
        \int_{B(y,r)} \,d\Gamma\Span{\xi} \leq C\left( \frac{r}{R} \right)^\delta \frac{\mu(B(y,r))}{\Psi(r)}
    \end{equation}
    for all $y \in X$ and $0 < r \leq 3R$.
\end{definition}

\begin{remark}\label{rem:Holder}
    If $(\mathcal{E},\mathcal{F})$ satisfies the Poincar\'e inequality \ref{eq:PI} and the cutoff energy condition \ref{eq:CE}, it then follows from \cite[Corollary 4.24]{anttila2025approach} that the cutoff functions provided by \ref{eq:CE} are necessarily Hölder continuous.
\end{remark}

We need the following two results from \cite{anttila2025approach} in our proofs.

\begin{lemma}[Theorem 4.9 \cite{anttila2025approach}]\label{lemma:charHK}
    Assume that $(X,d)$ is a complete geodesic metric space and that $\mu$ is a doubling measure on $(X,d)$. Then, for a given scale function $\Psi$, a strongly local regular Dirichlet form $(\mathcal{E},\mathcal{F})$ on $L^2(X,\mu)$ satisfies \ref{eq:HKMain} if and only if it satisfies both \ref{eq:PI} and \ref{eq:CE} for some $\delta > 0$.
\end{lemma}

\begin{lemma}[Proposition 4.10 \cite{anttila2025approach}]
\label{lemma:CS}
    Assume that $(X,d)$ is a complete geodesic metric space and that $\mu$ is a doubling measure on $(X,d)$. Then, for a given scale function $\Psi$, if a strongly local regular Dirichlet form $(\mathcal{E},\mathcal{F})$ on $L^2(X,\mu)$ satisfies \ref{eq:HKMain} then it also satisfies the following cutoff Sobolev inequality \ref{eq:CS}. For all $x \in X$ and $R \in (0,\diam(X))$ there is a cutoff function $\xi \in \mathcal{F}$ for $B(x,r) \subseteq B(x,2r)$ such that for every $y \in X$, $0 < r \leq 3R$ and $f \in \mathcal{F}$,
    \begin{equation}\label{eq:CS}
    \tag{$\textup{CS}_{\delta}(\Psi)$}
        \int_{B(y,r)} \tilde{f}^2 \, d\Gamma\Span{\xi} \leq C\left( \frac{r}{R} \right)^\delta \left( \int_{B(y,2r)} \, d\Gamma\Span{f} + \frac{1}{\Psi(r)} \int_{B(y,2r)}f^2 \, d\mu \right).
    \end{equation}
    Here $\tilde{f}$ is any quasicontinuous $\mu$-representative of $f$; see \cite[Chapter 2]{FOT} for a detailed account on quasicontinuity.
\end{lemma}

Throughout the paper, we will abuse notation and take a quasi-continuous $\mu$-representative of a given $f \in \mathcal{F}$ without writing $\tilde{f}$.

\section{Auxiliary lemmas}\label{sec:aux}
For the remainder of the work, we consider a fixed metric measure space $(X,d,\mu)$ such that $(X,d)$ is complete and geodesic, and $\mu$ is a doubling Radon measure on $(X,d)$. We fix a strongly local regular Dirichlet form $(\mathcal{E},\mathcal{F})$ on $L^2(X,\mu)$, and an inner uniform domain $\Omega \subseteq X$.
For simplicity, we denote the path metric of $\Omega$ by $\rho$, the completion of $(\Omega,\rho)$ by $(\widetilde{\Omega},\rho)$, and its Radon measure $A \mapsto \mu(\Omega \cap A)$, $A \subseteq \widetilde{\Omega}$, also by $\mu$.
Similarly, we understand the measures $\Gamma^\Omega\Span{\cdot}$ from Definition \ref{def:LocDir} as measures on $(\widetilde{\Omega},\rho)$ where $A \mapsto \Gamma^\Omega\Span{f}(A \cap \Omega)$ for all Borel sets $A \subseteq \widetilde{\Omega}$.
The open balls in $\widetilde{\Omega}$ are denoted
\[
    D(x,r) := \{ y \in \widetilde{\Omega} : \rho(x,y) < r \} \text{ for all } y \in \widetilde{\Omega} \text{ and } r > 0,
\]
whereas the open balls in $(X,d)$ are noted $B(x,r)$ like in the previous section.
Lastly, we fix a scale function $\Psi$ and assume $(\mathcal{E},\mathcal{F})$ to satisfy \ref{eq:HKMain}.

\begin{remark}\label{rem:B=D}
    The topologies on $\Omega$ induced by the metrics $d$ and $\rho$ coincide. Indeed, since $(X,d)$ is geodesic and $\Omega$ is open,
    \begin{equation}\label{eq:B=D}
        B(x,r) = D(x,r) \cap \Omega \text{ for all $x \in \Omega$ and $r \in (0,\dist(x,X \setminus \Omega))$.}
    \end{equation}
    Note that the inclusion $D(x,r) \cap \Omega \subseteq B(x,r)$ holds in general because $d \leq \rho$.
\end{remark}

\begin{remark}\label{rem:L2(Omega)}
    Since, by definition, $\mu(\widetilde{\Omega} \setminus \Omega) = 0$, we may identify $L^2(\widetilde{\Omega},\mu) = L^2(\Omega,\mu)$. In particular, we naturally regard $\mathcal{F}(\Omega) \subseteq L^2(\widetilde{\Omega},\mu)$.
    While this might seem artificial in the first glance, we also note that, for our goals, it is fairly natural to have $\mu(\widetilde{\Omega} \setminus \Omega) = 0$. Indeed, since $\Omega$ is an inner uniform domain, it follows fairly easily that  $\widetilde{\Omega} \setminus \Omega \subseteq \widetilde{\Omega}$ is a porous subset. Thus, by the Lebesgue differentiation theorem, it holds for any doubling measure $\nu$ on $(\widetilde{\Omega},\rho)$ that $\nu(\widetilde{\Omega} \setminus \Omega) = 0$.
\end{remark}

We first review some properties of inner uniform domains. These are already known in the literature. We provide the details because some are needed in later proofs.

\begin{lemma}[Lemma 3.9 \cite{SaloffCoste2011NeumannAD}]\label{lemma:IUD-doubling}
    The measure $\mu$ is a doubling measure on both $(\Omega,\rho)$ and $(\widetilde{\Omega},\rho)$.
    In particular, $(\widetilde{\Omega},\rho)$ is a proper metric space.
\end{lemma}

\begin{proof}
    We only consider the case of $(\widetilde{\Omega},\rho)$.
    Let $x \in \widetilde{\Omega}$ and $r > 0$. We may assume that $D(x,2r) \neq D(x,r)$, meaning there is $y \in (D(x,2r) \setminus D(x,r)) \cap \Omega$. Also fix $x_0 \in D(x,r/4) \cap \Omega$.
    Now, take a curve $\gamma \subseteq \Omega$ connecting $x_0$ to $y$ and which is provided by Definition \ref{def:IUD}. Then take a point $z \in \gamma$ such that $\rho(x_0,z) = r/4$.
    By using the properties of $\gamma$,
    \begin{align*}
        \dist(z,X \setminus \Omega) & \geq c_0 \frac{\rho(x_0,z) \rho(y,z)}{\rho(x_0,y)}= c_0\frac{r}{4} \frac{\rho(y,z)}{\rho(x_0,y)}\\
        & \geq c_0 \frac{r}{4} \frac{ \rho(x_0,y) - \rho(x_0,z)}{\rho(x,x_0) + \rho(x,y)} \\
        & \geq \frac{c_0}{12}(\rho(x_0,y)- \rho(x_0,z))\\
        & \geq \frac{c_0}{12}\left( \frac{3}{4}r -\frac{1}{4}r \right) = \frac{c_0}{24}r.
    \end{align*}
    Thus, according to \eqref{eq:B=D}, $D(z, \kappa r) \cap \Omega = B(z, \kappa r)$ for all $\kappa \in (0,c_0/24)$.
    Furthermore, we have the chain of inclusions
    \begin{align*}
        D(x,2r) \cap \Omega & \subseteq D(x_0,4r) \cap \Omega && (\rho(x,x_0) \leq r/4) \\
        & \subseteq D(z,8r) \cap \Omega && (\rho(x_0,z) = r/4) \\
        & \subseteq B(z,8r). && (d \leq \rho)
    \end{align*}
    By choosing $\kappa = (c_0/24 \land 1/4)$,
    \begin{align*}
        \mu(D(x,2r)) & \leq \mu(B(z,8r))\\
        & \leq C(D,c_0) \mu(B(z,\kappa r)) && (\text{Doubling property } \eqref{eq:mudoubling})\\
        & = C(D,c_0) \mu(D(z,\kappa r)) && (D(z, \kappa r) \cap \Omega = B(z, \kappa r))\\
        & \leq C(D,c_0) \mu(D(x,r)). && (D(z,\kappa r) \subseteq D(x,r))\\
    \end{align*}
    This concludes the proof.
\end{proof}

\begin{remark}
    What we proved in the previous lemma is that $\mu(D(x,r)) \approx \mu(B(x,r))$.
\end{remark}

The following result and its proof is originally by Väisälä \cite{vaisala1998}. See also \cite[Lemma 3.8]{Lierl2014}.

\begin{lemma}[Theorem 3.4 \cite{vaisala1998}]\label{lemma:vaisala}
    Let $x \in \Omega$, $r > 0$ and $\mathcal{C} \subseteq \Omega$ be the connected component of the subset $B(x,r) \cap \Omega$ containing $x$. Then there is a constant $\tau > 1$ depending only on the doubling constant of $\mu$ in \eqref{eq:mudoubling} and the constants $c_0,\,C_0$ in Definition \ref{def:IUD} such that
    \[
        D(x,r) \cap \Omega \subseteq \mathcal{C} \subseteq D(x,\tau r) \cap \Omega.
    \]
\end{lemma}

\begin{proof}
    First, the inclusion $D(x,r) \cap \Omega \subseteq \mathcal{C}$ is obvious because $D(x,r) \cap \Omega$ is path connected and $d \leq \rho$.
    Thus, let $y \in \mathcal{C}$ and we need to show that $y \in D(x,\tau r) \cap \Omega$ for some quantitative constant $\tau > 1$.
    Note that, under the current topological assumptions, $\mathcal{C}$ is path connected, meaning we can choose a curve $\gamma$ connecting $x$ to $y$ and contained in $\mathcal{C}$. We, nevertheless, cannot at the moment justify $\gamma$ to satisfy the conditions in Definition \ref{def:IUD}, but we can say that its diameter in the ambient metric $d$ has the bound $M := \diam(\gamma) \leq 2r$.

    By the compactness and continuity of $\gamma$, there are $x_1,\ldots, x_T \in \gamma$ such that $d(x_{i},x_{i+1}) = \rho(x_{i},x_{i+1})$ for all $i = 1,\ldots, T-1$ and $x = x_1,\, x_T = y$. Moreover, since $\max\{ \rho(x,x_i),\rho(y,x_i) \} \geq M/2$ for all $i$, the following holds by the inner uniformity of $\Omega$ and the first display in the proof of Lemma \ref{lemma:IUD-doubling}. There are $y_1,\ldots, y_T \in \Omega$ and curves 
    $\gamma_1,\ldots, \gamma_T \subseteq \Omega$ connecting $x_i$ to $y_i$ such that $\len(\gamma_i) \leq C_0 M$ and $\dist(y_i,X \setminus \Omega) \geq M c_0/24$.
    Now, we take the union
    \[
        U:= \bigcup_{i = 1}^T B(y_i,Mc_0/24) = \bigcup_{i = 1}^T D(y_i,M c_0/24) \subseteq \Omega,
    \]
    and denote its distinct connected components $\mathcal{C}_1,\ldots, \mathcal{C}_P \subseteq U$.
    The next step is to find suitable uniform estimates for the geometry of $U$.
    
    For each $k = 1,\ldots,P$, we fix a point $y_{i_k} \in \{y_i\}_{i = 1}^T$ such that $B(y_{i_k},Mc_0/24) \subseteq \mathcal{C}_k$, which is possible because open balls in the present setting are connected.
    Moreover, the balls $B(y_{i_k},c_0M/4)$ $k = 1,\ldots,P$ are pairwise disjoint because the components are pairwise disjoint.
    Since,
    \[
    U \subseteq B(x,(1 + C_0 + c_0/24)M)
    \]
    it now follows from the metric doubling property that $P$ has an upper bound $P_0 = P_0(c_0,D)$ depending only on $c_0$ the doubling constant $D$ in \eqref{eq:mudoubling}.

    Then, we fix any $k = 1,\ldots, P$ and take two distinct open balls
    \[
    B(y_j,c_0 M/24),\,B(y_l,c_0 M/24) \subseteq \mathcal{C}_k
    \]
    for $j,l \in \{1,\ldots T\}$.
    We note that there is a sequence $z_1,\ldots, z_{L+1} \subseteq \{ y_i \}_{i = 1}^T$ such that $B(z_i,c_0M/24) \cap B(z_{i+1},c_0M/24) \neq \emptyset$. If such sequence would not exist, then $\mathcal{C}_k$ would be disconnected. By taking the length of this sequence, namely $L$, to be minimal, we see that the collection of open balls $\{ B(z_{2i},Mc_0/24) \}_{i = 1}^{\floor{L/2}}$ are pairwise disjoint. Thus, it follows from the metric doubling property and the argument above that $L$ has an upper bound $L_0 = L_0(c_0,D)$ depending only on $c_0$ and $D$ in \eqref{eq:mudoubling}.

    We have now gathered all the required ingredients to produce a rectifiable curve $\theta \subseteq \Omega$ connecting $x$ to $y$ with the length estimate $\len(\theta) \leq C(c_0,C_0,D) \cdot M$.
    First, after a reordering and removing some of the connected components of $U$, the following conditions hold.
    \begin{enumerate}
        \vspace{2pt}
        \item There are continuous curves $\theta_1,\, \theta_{P} \subseteq \Omega$ connecting $x$ to $\mathcal{C}_1$ and $y$ to $\mathcal{C}_P$, respectively, with $\len(\theta_1),\, \len(\theta_{P}) \leq C_0 M$.
        We simply choose $\theta_1 := \gamma_1$ and $\theta_{P} := \gamma_T$.
        \vspace{2pt}
        \item For each $k = 2,\ldots, P-1$ there is a continuous curve $\theta_k$ connecting $\mathcal{C}_k$ to $\mathcal{C}_{k+1}$ with $\len(\theta_k) \leq (2C_0 + 1)M$.
        Specifically, for a suitable choice of $i \in \{1,\ldots,T\}$, we choose $\theta_k$ by the concatenation of the three curves, $\gamma_{i}$ connecting $\mathcal{C}_k$ to $x_i$, a shortest curve $\gamma_{[x_i,x_{i+1}]}$ connecting $x_{i}$ to $x_{i+1}$ and $\gamma_{i+1}$ connecting $x_{i+1}$ to $\mathcal{C}_{k+1}$. Note that, to ensure that $\gamma_{[x_i,x_{i+1}]} \subseteq \Omega$, we use 
        $\rho(x_i,x_{i+1}) = d(x_i,x_{i+1}) \leq M$.
    \end{enumerate}
    Next, for each $k = 1,\ldots P-1$, we let $\tilde{\theta}_k \subseteq U \subseteq \Omega$ to be a curve that connects the endpoint of $\theta_k$ to the starting point of $\theta_{k+1}$ with the length of at most $L_0 c_0 M/12$. The existence of such curves follows the fact that
    $B(y_i,Mc_0/24) = D(y_i,Mc_0/24)$ for each $i$.
    By taking the concatenation of the curves $\{ \theta_k \}_{k = 1}^{P}$ and $\{ \tilde{\theta}_k \}_{k = 1}^{P-1}$,
    we obtain a curve $\theta$ connecting $x$ to $y$ with length
    \[
    \len(\theta) \leq P_0(2C_0 + 1 + L_0c_0/12)M.
    \]

    Finally, by choosing $\tau := 2P_0(2C_0 + 1 + L_0c_0/12)$, we have $y \in D(x,\tau r)$, and this completes the proof.
\end{proof}

We recall the following functional theoretic property of Dirichlet forms. For references, see for instance \cite[Corollary 1.1.9]{chen2012symmetric} and \cite[Proposition 3.18]{KajinoContraction}.
Note that the Clarkson's inequality used in \cite{KajinoContraction} follows from the parallelogram law.
We, nevertheless, provide the details since similar arguments are used later in the work.

\begin{lemma}\label{lemma:FunAnal}
    The Dirichlet form $(\mathcal{E},\mathcal{F})$ on $L^2(X,\mu)$ is lower-semicontinuous in the following sense.
    Given any sequence $\{f_n\}_{n = 1}^\infty \subseteq \mathcal{F}$ such that $\sup_{n \in \N} \mathcal{E}(f_n) < \infty$ and $f_n \to f$ in $L^2(X,\mu)$, then $f \in \mathcal{F}$ and
    \[
        \mathcal{E}(f) \leq \liminf_{n \to \infty} \mathcal{E}(f_n).
    \]
    If additionally $\lim_{n \to \infty} \mathcal{E}(f_n) = \mathcal{E}(f)$, then $f_n \to f$ in $(\mathcal{F},\mathcal{E}_1)$.
\end{lemma}

\begin{proof}
    Let $\{f_n\}_{n = 1}^\infty \subseteq \mathcal{F}$ and $f \in \mathcal{F}$ be as in the claim.
    We also fix a subsequence $\{ f_{n_k} \}_{k = 1}^\infty \subseteq \{f_n\}_{n = 1}^\infty$ such that
    \[
    \lim_{k \to \infty} \mathcal{E}(f_{n_k}
    ) = \liminf_{n \to \infty} \mathcal{E}(f_{n}).
    \]
    Recall that it follows from the Riesz representation theorem that every Hilbert space is reflexive. Since the sequence $\{ f_n \}_{n = 1}^\infty \subseteq \mathcal{F}$ is bounded, by taking another subsequence if necessary, we may assume the weak convergence $f_{n_k} \rightharpoonup g$ as $k \to \infty$ for some $g \in \mathcal{F}$.
    Then, it follows from Mazur's lemma; see \cite[Chapter V]{Yosida}, that some convex combinations of the form
    \[
        h_k := \sum_{l = k}^{N_k} \lambda_{k,l} f_{n_k} \in \mathcal{F}
    \]
    converge strongly to $g$ in $(\mathcal{F},\mathcal{E}_1)$. Since we necessarily have $h_k \to f$ and $h_k \to g$ in $L^2(X,\mu)$, this implies $f = g \in \mathcal{F}$.

    We verify the lower-semicontinuity property.
    By the Cauchy--Schwarz inequality,
    \[
    \left|\mathcal{E}_1\left(f_{n_k},f/\sqrt{\mathcal{E}_1(f)}\right)\right|^2 \leq \mathcal{E}_1(f_{n_k}),
    \]
    where we understand $f/\sqrt{\mathcal{E}_1(f)} = 0$ if $f = 0$.
    Since $f_{n_k} \rightharpoonup f$,
    \[
        \mathcal{E}_1(f) = \lim_{k \to \infty} \left|\mathcal{E}_1\left(f_{n_k},f/\sqrt{\mathcal{E}_1(f)}\right)\right|^2 \leq \lim_{k \to \infty} \mathcal{E}_1(f_{n_k}) = \liminf_{n \to \infty} \mathcal{E}_1(f_n).
    \]
    By $f_n \to f$ in $L^2(X,\mu)$, the desired lower-semicontinuity follows by subtracting $\int_X f^2 \, d\mu$ from the previous display.

    Lastly, we assume $\lim_{n \to \infty} \mathcal{E}(f_n) = \mathcal{E}(f)$. By the parallelogram law,
    \[
        \mathcal{E}(f-f_n) = 2\mathcal{E}(f) + 2 \mathcal{E}(f_n) - \mathcal{E}(f + f_n).
    \]
    By the lower-semicontinuity property and the triangle inequality,
    \begin{align*}
        2\sqrt{\mathcal{E}(f)} & \leq 
        \liminf_{n \to \infty} \sqrt{\mathcal{E}(f + f_n)} \leq 
        \limsup_{n \to \infty} \sqrt{\mathcal{E}(f + f_n)}\\
        & \leq \lim_{n \to \infty} \sqrt{\mathcal{E}(f)} + \sqrt{\mathcal{E}(f_n)} = 2\sqrt{\mathcal{E}(f)}.
    \end{align*}
    The combination of the previous two displays implies $ \mathcal{E}(f-f_n) \to 0$ as $n \to \infty$.
\end{proof}

\begin{remark}\label{rem:CS}
    If $\xi \in \mathcal{F}$ is a cutoff function for $B(x,r) \subseteq B(x,2r)$ provided by \ref{eq:CS}, then $f \cdot \xi \in \mathcal{F}$ for all $f \in \mathcal{F}$. To see this, first fix $n \in \N$ and $f_n := (f \land n) \lor -n.$ Then, $f_n \cdot \xi \in \mathcal{F}$ because $\mathcal{F} \cap L^\infty(X,\mu)$ is an algebra. By applying the Leibniz rule \cite[Lemma 3.25]{FOT}, strong locality \cite[Corollary 3.2.1]{FOT} the inequality of measures $\Gamma\Span{f_n} \leq \Gamma\Span{f}$ \cite[Equation (3.2.16)]{FOT}, and the cutoff Sobolev inequality \ref{eq:CS},
    \begin{align*}
        \mathcal{E}(f_n \cdot \xi) & \leq 2\left(\int_{X} f_n^2 \, d\Gamma\Span{\xi} + \int_{X} \xi^2 \, d\Gamma\Span{f_n}\right)\\
        & \leq 2\left(\int_{B(x,2r)} f^2 \, d\Gamma\Span{\xi} + \int_{B(x,2r)} \,d\Gamma\Span{f}\right) \\
        & \leq A\left(\int_{B(x,2r)} \, d\Gamma\Span{f} + \frac{1}{\Psi(r)} \int_{B(x,2r)} f^2\, d\mu\right).
    \end{align*}
    Thus, $\{ f_n \cdot \xi \}_{i = 1}^\infty \subseteq \mathcal{F}$ is a bounded sequence that converges in $L^2(X,\mu)$.
    It then follows from Lemma \ref{lemma:FunAnal} that $f \cdot \xi \in \mathcal{F}$ and
    \[
        \mathcal{E}(f \cdot \xi) \leq A\left(\int_{B(x,2r)} \, d\Gamma\Span{f} + \frac{1}{\Psi(r)} \int_{B(x,2r)} f^2\, d\mu\right).
    \]
\end{remark}

The proof of the following lemma is somewhat technical because we use the cutoff Sobolev inequality; see \cite[Lemma 4.4]{kajino2023conformal} and \cite[Proposition 2.50]{SaloffCoste2011NeumannAD} for similar arguments.
However, we note that there is a different  method, which is also more involved, for similar purpose that does not require such a technical condition. For further details, see \cite[Theorems 6.2.4-6.2.5]{chen2012symmetric}.

\begin{lemma}\label{lemma:ReflDiff}
$(\mathcal{E}^\Omega,\mathcal{F}(\Omega))$ is a strongly local Dirichlet form on $L^2(\widetilde{\Omega},\mu)$.
\end{lemma}

\begin{remark}
    We do not yet consider the regularity of $(\mathcal{E}^\Omega,\mathcal{F}(\Omega))$.
\end{remark}

\begin{proof}[Proof of Lemma \ref{lemma:ReflDiff}]
    We first show that $(\mathcal{E}^\Omega,\mathcal{F}(\Omega))$ is a Dirichlet form on $L^2(\widetilde{\Omega},\mu)$.
    According to the discussion in Remark \ref{rem:L2(Omega)}, we may as well prove $(\mathcal{E}^\Omega,\mathcal{F}(\Omega))$ to be a Dirichlet form on $L^2(\Omega,\mu)$; note that the difference in the topologies of $(\Omega,\rho)$ and $(\widetilde{\Omega},\rho)$ is not relevant in the first two conditions of Definition \ref{def:Dir}.
    Throughout the proof, we consider a fixed sequence of relatively compact subsets $V_1 \Subset V_2 \Subset \cdots \Subset \Omega$ such that $\Omega = \bigcup_{i = 1}^\infty V_i$, which exists because $(X,d)$ is proper.
    
    We first verify the Markov property, Definition \ref{def:Dir}-(2). 
    Let $f \in \mathcal{F}(\Omega)$, $V \Subset \Omega$ be a relatively compact open subset and $f^{\#} \in \mathcal{F}$ such that $f\mathds{1}_{V} = f^{\#}\mathds{1}_{V}$ $\mu$-almost everywhere.
    Then $(f^+ \land 1)\mathds{1}_{V} = ((f^{\#})^+ \land 1)\mathds{1}_{V}$ $\mu$-almost everywhere and $(f^{\#})^+ \land 1 \in \mathcal{F}$ by the Markov property of $(\mathcal{E},\mathcal{F})$. Therefore, $f^+ \land 1 \in \mathcal{F}_{\loc}(\Omega)$.
    By using the fact that $\Gamma\Span{g^+\land 1} \leq \Gamma\Span{g}$ for all $g \in \mathcal{F}$; see \cite[Equation (3.2.16)]{FOT},
    \begin{align*}
        \mathcal{E}^\Omega(f^+\land 1) = \lim_{i \to \infty} \Gamma^\Omega\Span{f^+\land 1}(V_i) \leq \lim_{i \to \infty} \Gamma^\Omega\Span{f}(V_i) = \mathcal{E}^\Omega(f).
    \end{align*}
    This completes the proof of the Markov property.
    
    Next, we verify Definition \ref{def:Dir}-(1). Note that, for all Borel sets $A \subseteq X$, the mapping $(f,g) \mapsto \Gamma\Span{f,g}(A)$, $\mathcal{F} \times \mathcal{F} \to \R$, is symmetric, non-negative definite and bilinear. This can be seen from \cite[Equation (3.2.16)]{FOT}.
    By using a similar argument as for the Markov property, it now follows that $(f,g) \mapsto \mathcal{E}^\Omega(f,g)$, $\mathcal{F}(\Omega) \times \mathcal{F}(\Omega) \to \R$, also satisfies these properties. Moreover, the density $\mathcal{F}(\Omega) \subseteq L^2(\Omega,\mu)$ easily follows from the density $\mathcal{F} \subseteq L^2(X,\mu)$ by restricting functions in $\mathcal{F}$ to $\Omega$.
    Thus, we need to check the completeness of $(\mathcal{F}(\Omega),\mathcal{E}_1^\Omega)$.
    
    To this end, fix a Cauchy sequence $\{f_i\}_{i = 1}^\infty \subseteq \mathcal{F}(\Omega)$ and let $f \in  L^2(\Omega,\mu)$ be its limit in the $L^2$-norm. We need to prove that, for every relatively compact open subset $V\Subset \Omega$, there is $f^{\#} \in \mathcal{F}$ such that $f\mathds{1}_V = f^\# \mathds{1}_V$.
    Thus, fix such $V \Subset\Omega$. We also take a finite covering $V \subseteq \bigcup_{j \in J} B(x_j,r_j)$ such that $B(x_j,5r_j) \subseteq \Omega$, and a second relatively compact open subset $U_V := \bigcup_{i \in I}B(x_j,4r_j) \Subset \Omega$. For each $j \in J$ let $\xi_j \in \mathcal{F}$ be a cutoff function for $B(x_j,r_j) \subseteq B(x_j,2r_j)$ be provided by the cutoff Sobolev inequality \ref{eq:CS}, which holds by Lemma \ref{lemma:CS}.
    Then, set $\xi_V := (\sum_{j \in J} \xi_j) \land 1 \in \mathcal{F}$.
    
    Next, fix functions $f_i^{\#} \in \mathcal{F}$ such that $f_i\mathds{1}_{U_V} = f_i^{\#}\mathds{1}_{U_V}$ $\mu$-almost everywhere, and consider $f_i^{\#} \cdot \xi_V$. According to Remark \ref{rem:CS}, $f_i^{\#} \cdot \xi \in \mathcal{F}$ and
    \begin{equation*}
        \mathcal{E}(f_i^{\#} \cdot \xi_V - f_k^{\#} \cdot \xi_V) \lesssim \sum_{j \in J} \int_{B(x_j,2r_j)} \, d\Gamma\Span{f_i - f_k} + \frac{1}{\Psi(r_j)} \int_{B(x_j,2r_j)} (f_i - f_k)^2\, d\mu.
    \end{equation*}
    Moreover, clearly $f_i^{\#} \cdot \xi_V \to f \cdot \xi_V$ in $L^2(X,\mu)$. Therefore, it follows from the previous display that $\{ f_i^{\#} \cdot \xi_V \}_{i = 1}^\infty \subseteq \mathcal{F}$ is a Cauchy sequence, meaning $f \cdot \xi \in \mathcal{F}$ since $(\mathcal{F},\mathcal{E}_1)$ is a Hilbert space.
    Also $(f \cdot \xi_V)\mathds{1}_V = f\mathds{1}_V$. Because $V \Subset \Omega$ is arbitrary, $f \in \mathcal{F}_{\loc}(\Omega)$. 

    Next, we show that $\mathcal{E}^\Omega(f) < \infty$.
    By the previous part of the proof, for each $l \in \N$,
    $\Gamma\Span{f_i - f}(V_l) \to 0$ as $i \to \infty$.
    Using this,
    \begin{equation*}\label{eq:Cauchy}
        \Gamma^\Omega\Span{f}(V_l) \leq \limsup_{i \to \infty} \left\{\Gamma^\Omega\Span{f-f_i}(V_l) + \Gamma^\Omega\Span{f_i}(V_l)\right\} \leq \limsup_{i \to \infty} \mathcal{E}^\Omega(f_i).
    \end{equation*}
    By letting $l \to \infty$, we get $\mathcal{E}^\Omega(f) \leq \limsup_{i \to \infty} \mathcal{E}^\Omega(f_i)$, and the upper bound is finite because $\{f_i\}_{i = 1}^\infty \subseteq \mathcal{F}(\Omega)$ is a Cauchy sequence.
    Hence, $\mathcal{F}(\Omega)$ equipped with the inner product $\mathcal{E}^\Omega_1$ is a Hilbert space.

    Lastly, we show that $(\mathcal{E}^\Omega,\mathcal{F}(\Omega))$ is strongly local on $L^2(\widetilde{\Omega},\mu)$. Let $f,g \in \mathcal{F}(\Omega)$ such that the assumptions in Definition \ref{def:Dir}-(4) hold for the constant $a \in \R$.
    It is sufficient to show that $\Gamma^\Omega\Span{f,g}(V_l) = 0$ for all $l \in \N$. Fix $l \in \N$ and let $\xi_{V_l}$ and $U_{V_l}$ be as in the earlier part of the proof for $V = V_l$.
    Also let $f^{\#},g^{\#} \in \mathcal{F}$ such that $f\mathds{1}_{U_{V_l}} = f^{\#}\mathds{1}_{U_{V_l}}$ and $g\mathds{1}_{U_{V_l}} = g^{\#}\mathds{1}_{U_{V_l}}$.
    Then $f^{\#} \cdot \xi_{V_l},\, g^{\#} \cdot \xi_{V_l} \in \mathcal{F}$ have compact supports and there is an open set $V_l \subseteq U \Subset U_{V_l}$ such that the following holds.
    \begin{enumerate}
        \vspace{2pt}
        \item $\supp_{\mu}[f^{\#} \cdot \xi_V] \cap V_l \subseteq U$.
        \vspace{2pt}
        \item $g^{\#} \cdot \xi_{V}$ is constant on $U$.
    \end{enumerate}
    By these properties, it follows from the strong locality of $(\mathcal{E},\mathcal{F})$ along with the formula determining the energy measures, displayed in Definition \ref{def:Dir}, and the density $\mathcal{F} \cap C_c(X) \subseteq C_c(X)$ in the uniform norm,
    \[
    \Gamma^\Omega\Span{f,g}(V_l) = \Gamma\Span{f^{\#} \cdot \xi_{V_l}, g^{\#} \cdot \xi_{V_l}}(V_l) = 0.
    \]
    This completes the proof.
\end{proof}

\begin{lemma}\label{lemma:Leibniz}
    The following equalities of signed measures hold.
    \begin{enumerate}
        \vspace{2pt}
        \item If $f_1,f_2,f_3 \in \mathcal{F}_{\loc}(\Omega) \cap C(\widetilde{\Omega})$,
        \[
            d\Gamma^\Omega\Span{f_1 \cdot f_2,f_3} = f_1 \cdot d \Gamma^\Omega\Span{f_2,f_3} + f_2 \cdot d \Gamma^\Omega\Span{f_1,f_3}.
        \]
        \item If $f_1,f_2 \in \mathcal{F}_{\loc}(\Omega) \cap C(\widetilde{\Omega})$ and $\psi : \R \to \R$ is a smooth function with $\psi(0) = 0$,
        \[
            d\Gamma^\Omega\Span{ \psi \circ f_1,f_2} = (\psi' \circ f_1)\cdot d\Gamma^\Omega\Span{f_1,f_2}.
        \]
    \end{enumerate}
\end{lemma}

\begin{remark}\label{rem:QS}
    We have restricted the statements in (1) and (2) in Lemma \ref{lemma:Leibniz} to continuous functions to avoid the somewhat technical detail related to quasicontinuity in later arguments.
\end{remark}

\begin{proof}[Proof of Lemma \ref{lemma:Leibniz}]
    The prove the condition (1), it is sufficient to show that, for any relatively compact open subset $V \Subset \Omega$ and every relatively compact subset $A \Subset V$ the signed Radon measures in the claim agree on $A$.
    The desired equality then follows from the Caratheodory extension theorem.
    
    To this end, take such $A$ and $V$, and fix $f_1,f_2,f_3 \in \mathcal{F}_{\loc}(\Omega) \cap C(\widetilde{\Omega})$.
    Take any $f_1^{\#},f_2^{\#},f_3^{\#} \in \mathcal{F} \cap C_c(X)$ such that $f_i = f_i^{\#}$ holds point-wise on $V$.
    To see that we can choose them in $C_c(X)$, we use the fact that $f_i$ are continuous in an open neighborhood $U$ of $V \Subset U \Subset \Omega$ and use the cutoff function $\xi_V$ in an analogous manner as in the proof of Lemma \ref{lemma:ReflDiff}.
    Since $f_i^{\#}$ are continuous, they are in particular quasicontinuous, and the Leibniz rule \cite[Lemma 3.2.5]{FOT} now implies the equality of signed Radon measures
    \[
            d\Gamma\Span{f_1^{\#} \cdot f_2^{\#},f_3^{\#}} = f_1^{\#} \cdot d \Gamma\Span{f_2^{\#},f_3^{\#}} + f_2^{\#} \cdot d \Gamma\Span{f_1^{\#},f_3^{\#}}.
    \]
    Since the equalities $f_i^{\#} = f_i$ hold point-wise on $V$, we see that the Radon measures in the claim agree on $A$. This completes the proof of (1). We note that, if we wish to extend this argument to the case where $f_i$ are not necessarily continuous, we need to perform a somewhat delicate choices of $\mu$-representatives.
    
    The condition (2) is worked out in a similar manner with the difference that the Leibniz rule is replaced with the chain rule \cite[Theorem 3.2.2]{FOT} of $(\mathcal{E},\mathcal{F})$,
    \[
        d\Gamma\Span{ \psi \circ f_1^{\#},f_2^{\#}} = (\psi' \circ f_1^{\#})\cdot d\Gamma\Span{f_1^{\#},f_2^{\#}},
    \]
    where $\psi$ is as in the claim.
\end{proof}

Lastly, we state the validity of a Poincar\'e type inequality. We note that this is not yet the Poincar\'e inequality in Definition \ref{def:PI(Dir)} because we have not yet verified that $\Gamma^\Omega\Span{\cdot}$ are the energy measures of $(\mathcal{E}^\Omega,\mathcal{F}(\Omega))$ in the sense of Definition \ref{def:Dir}.
Indeed, we need the regularity for the energy measures to be well-defined in the first place.

\begin{lemma}\label{lemma:PIRefl}
    There are constants $C,\sigma \geq 1$ such that the following holds. For all $x \in \widetilde{\Omega}$, $r > 0$ and $f \in \mathcal{F}(\Omega)$,
    \[
    \int_{D(x,r)} (f - f_{D(x,r)})^2 \, d\mu \leq C \Psi(r) \int_{D(x,\sigma r)} \, d\Gamma^\Omega\Span{f}.
    \]
\end{lemma}

\begin{proof}
    This has been proven in the Gaussian case ($\Psi(r) = r^2$) using a Whitney covering argument \cite[Theorem 3.12]{SaloffCoste2011NeumannAD}. See also an analogous result with an almost identical proof in the sub-Gaussian case when $\Omega \subseteq X$ is a uniform domain \cite[Subsection 5.1]{murugan2024heat}. The same proof works in the present setting as well, and we thus omit the details here.
\end{proof}

\section{Main theorem}\label{sec:MainThm}

This section proves the main theorem of the paper, and we work with the same notation and assumptions as discussed in the beginning of Section \ref{sec:aux}.
We first state a more general version of the main theorem stated in Introduction.
\begin{theorem}[Theorem \ref{thm:Main}]\label{thm:MainMain}
    $(\mathcal{E}^\Omega,\mathcal{F}(\Omega))$ is a strongly local regular Dirichlet form on $L^2(\widetilde{\Omega},\mu)$ satisfying the heat kernel estimates \ref{eq:HKMain}.
\end{theorem}

\subsection{Regularity}
The first objective is to establish the regularity of $(\mathcal{E}^\Omega,\mathcal{F}(\Omega))$.
The proof heavily relies on the cutoff energy condition \ref{eq:CE} of $(\mathcal{E},\mathcal{F})$, which holds according to Lemma \ref{lemma:charHK}.

\begin{lemma}\label{lemma:cutoffs}
    There are constants $\delta > 0$ and $C \geq 1$ such that the following holds.
    For all $x \in \widetilde{\Omega}$ and $R \in (0,\diam(\widetilde{\Omega}))$ there is a cutoff function $\xi \in \mathcal{F}(\Omega)$ for $D(x,r) \subseteq D(x,2r)$ satisfying
    \[
        \int_{D(y,r) \cap \Omega} \, d\Gamma^\Omega\Span{\xi} \leq C \left( \frac{r}{R} \right)^{\delta} \frac{\mu(D(x,r))}{\Psi(r)}.
    \]
    for all $y \in \widetilde{\Omega}$ and $0 < r \leq 3R$.
\end{lemma}

\begin{proof}
    The main technical detail in the proof is to ensure that suitable restrictions of cutoff functions in the ambient space $X$ are cutoff functions in $\Omega$ with respect to the path metric $\rho$.
    To this end, we use the result of Väisälä, stated in Lemma \ref{lemma:vaisala}.
    
    Let $x \in \widetilde{\Omega}$ and $R \in (0,\diam(\widetilde{\Omega}))$. It is clear that we may assume $x \in \Omega$.
    Furthermore, it follows from a covering argument that it is sufficient to construct cutoff functions for $D(x,r) \subseteq D(x,\sigma r)$, $\sigma > 1$ is any constant independent of $x$ and $r$, such that the desired energy estimate holds.
    
    First, let $\xi_0 \in \mathcal{F}$ be a cutoff function for $B(x,r) \subseteq B(x,2r)$ provided by \ref{eq:CE}.
    We let $\mathcal{C}$ be the connected component of $B(x,4r) \cap \Omega$ containing $x$, and consider the Borel function $\xi$ defined on $\Omega$ according to
    \[
    \xi(z) := 
    \begin{cases}
        \xi_0(z) & \text{ if } z \in \mathcal{C}\\
        0 & \text{ if } z \in \Omega \setminus \mathcal{C}.
    \end{cases}
    \]
    We first show that $\xi \in \mathcal{F}_{\loc}(\Omega)$.
    To this end, note that it follows from the present topological assumptions that $\mathcal{C} \subseteq \Omega$ is an open subset. Also, by its definition, $\xi$ is identically zero in an open neighborhood of $\Omega \setminus \mathcal{C} \subseteq \Omega$. Hence, every point $x \in \Omega$ admits a relatively compact open neighborhood $x \in V_x \Subset \Omega$ such that $\xi \mathds{1}_{V_x} = f^\# \mathds{1}_{V_x}$ for some $f^\# \in \mathcal{F}$. It then follows from a similar argument as in the proof of Lemma \ref{lemma:ReflDiff} that $\xi \in \mathcal{F}_{\loc}(\Omega)$.
    Moreover, the strong locality implies $\mathcal{E}_1^\Omega(\xi) \leq \mathcal{E}_1(\xi_0) < \infty$, which proves $\xi \in \mathcal{F}(\Omega)$.
    Furthermore, since $\xi_0$ is Hölder continuous in the metric $d$ according to Remark \ref{rem:Holder} and $d \leq \rho$, $\xi : \Omega \to \R$ is Hölder continuous in the metric $\rho$. Therefore, $\xi$ extends continuously to the completion $\widetilde{\Omega}$, meaning $\xi \in \mathcal{F}(\Omega) \cap C_c(\widetilde{\Omega})$.

    The energy estimates now follows easily from the definitions of $\xi$ and $\Gamma^\Omega\Span{\cdot}$,
    \[
        \int_{D(y,r) \cap \Omega} \, d\Gamma^\Omega\Span{\xi} \leq \int_{B(y,r)} \, d\Gamma\Span{\xi_0} \lesssim \left( \frac{r}{R} \right)^\delta \frac{\mu(B(y,r))}{\Psi(r)} \lesssim \left( \frac{r}{R} \right)^\delta \frac{\mu(D(y,r))}{\Psi(r)}.
    \]
    The first inequality follows from the definition of $\xi$ and $D(y,r) \cap \Omega \subseteq B(y,r)$. The last inequality follows from $\mu(D(x,r)) \approx \mu(B(x,r))$ which can be seen from the proof of Lemma \ref{lemma:IUD-doubling}.

    Lastly, we show that $\xi$ is a cutoff function.
    If $\tau > 1$ is as in Lemma \ref{lemma:vaisala}, then $D(x,r) \cap \Omega \subseteq \mathcal{C} \subseteq D(x,2\tau r) \cap \Omega$. Therefore $\xi|_{\Omega \cap D(x,r)} = 1$ and $\xi|_{\Omega \setminus D(x,2\tau r)} = 0$.
    By its definition $0 \leq \xi \leq 1$, meaning $\xi$ is a cutoff function for $D(x,r) \cap \Omega \subseteq D(x,2\tau r) \cap \Omega$.
    Since $(\widetilde{\Omega},\rho)$ is a completion of the length space $(\Omega,\rho)$,
    it follows that the continuous extension of $\xi$ to $\widetilde{\Omega}$ is a cutoff function for $D(x,r) \subseteq D(x,2\tau r)$. The proof is completed by choosing $\sigma = 2\tau$.
\end{proof}

\begin{remark}
    A simplified version of the previous argument can be used to establish the cutoff energy condition when $\Omega$ is a uniform domain in the sense of \cite[Definition 2.3]{murugan2024heat}. Indeed, since the metric on the uniform domain $\Omega \subseteq X$ is simply the restriction of the ambient metric $d$, one may take the restriction of the cutoff functions from the ambient space directly, without the need to analyze the connected components.
    This observation answers another question of Murugan, posing whether the sub-Gaussian heat kernel estimates for reflected diffusion on uniform domains can be established without relying on the extension operator \cite[Section 6.3]{murugan2024heat}.
\end{remark}

\begin{corollary}\label{cor:Ucap}
    There is a constant $C \geq 1$ satisfying the following condition.
    For every $x \in X$ and $R \in (0,\diam(\widetilde{\Omega}))$
    there is a cutoff function $\varphi \in \mathcal{F}(\Omega)$ for $D(x,r) \subseteq D(x,2r)$ such that
    \[
    \mathcal{E}^\Omega(\varphi) \leq C \frac{\mu(D(x,r))}{\Psi(r)}.
    \]
\end{corollary}

\begin{proof}
    The claim follows from Lemma \ref{lemma:cutoffs} and the strong locality.
\end{proof}

\begin{corollary}\label{cor:CS}
    There are constants $\delta > 0$ and $C \geq 1$ such that the following holds.
    For all $x \in \widetilde{\Omega}$ and $R \in (0,\diam(\widetilde{\Omega}))$ there is a cutoff function $\xi \in \mathcal{F}(\Omega)$ for $D(x,r) \subseteq D(x,2r)$ satisfying
    \[
        \int_{D(y,r) \cap \Omega} f^2 \, d\Gamma^\Omega\Span{\xi} \leq C\left( \frac{r}{R} \right)^\delta \left( \int_{D(y,2r) \cap \Omega} \, d\Gamma^\Omega\Span{f} + \frac{1}{\Psi(r)} \int_{D(y,2r) \cap \Omega}f^2 \, d\mu \right).
    \]
    for all $y \in \widetilde{\Omega}$, $0 < r \leq 3R$ and $f \in \mathcal{F}(\Omega) \cap C(\widetilde{\Omega})$.
\end{corollary}

\begin{remark}
    The previous corollary is stated only for continuous $f$ for the same reason as discussed in Remark \ref{rem:QS}.
\end{remark}

\begin{proof}[Proof of Corollary \ref{cor:CS}]
    The cutoff functions $\xi$ for $D(x,r) \subseteq D(x,2r)$ in Lemma \ref{lemma:cutoffs} satisfy the desired energy inequality.
    This follows from the Poincar\'e inequality in Lemma \ref{lemma:PIRefl} and \cite[Proposition 3.4 and the proof of Proposition 4.10]{anttila2025approach}. We note that the \emph{precise representatives} used in Proposition \cite[Proposition 3.4]{anttila2025approach} for functions in $\mathcal{F}(\Omega) \cap C(\widetilde{\Omega})$ are the point-wise defined continuous representatives.
\end{proof}

\begin{proposition}\label{prop:SLRDF}
$(\mathcal{E}^\Omega,\mathcal{F}(\Omega))$ is a strongly local regular Dirichlet form on $L^2(\widetilde{\Omega},\mu)$.
\end{proposition}

\begin{proof}
    By Lemma \ref{lemma:ReflDiff}, we only need to verify the regularity. First, the density $\mathcal{F}(\Omega) \cap  C_c(\widetilde{\Omega}) \subseteq C_c(\widetilde{\Omega})$ in the uniform norm follows from standard arguments using Lemma \ref{lemma:cutoffs}. For instance, we could use Stone--Weierstrass theorem.

    Next, we verify the density $\mathcal{F}(\Omega) \cap C_c(\widetilde{\Omega}) \subseteq \mathcal{F}(\Omega)$ in the Hlbert space $(\mathcal{F}(\Omega),\mathcal{E}^\Omega_1)$.
    The method is a fairly standard partition of unity argument. The details are essentially identical to those in \cite[Theorem 3.12]{shimizu2025characterizations}.
    
    Let $f \in \mathcal{F}$.
    For each $n \in \N$ we fix a maximal $2^{-n}$ separated subsets $Z_n \subseteq X$, meaning
    \[
    d(v,w) \geq 2^{-n} \text{ for every pair of distinct } v,w\in Z_n
    \]
    and
    \[
       X = \bigcup_{v \in Z_n} B(v,2^{-n}).
    \]
    Note that $Z_n$ is at most countably infinite by the metric doubling property.
    
    Now, since the measure $\mu$ on $(\widetilde{\Omega},\rho)$ is doubling according to Lemma \ref{lemma:IUD-doubling}, it follows from Lemma \ref{lemma:Leibniz} and Corollary \ref{cor:Ucap} that we can construct a partition of unity using the standard construction; see \cite[Proof of Theorem 6.10]{EvansGariepy} for analogous techniques in the Euclidean spaces and also \cite[Section 2.3]{MuruganChain20} or
    \cite[Lemma 3.10]{shimizu2025characterizations} for similar arguments in more abstract settings.
    Indeed, there is a constant $P > 0$ depending only on the doubling constant of $\mu$, the constants in Definition \ref{def:IUD}, and the constants in Lemma \ref{lemma:PIRefl} and Corollary \ref{cor:Ucap}, such that for each $n \in \N$ there is a family of functions $\{ \varphi_{v,n} \}_{v \in Z_n} \subseteq \mathcal{F}(\Omega) \cap C_c(\widetilde{\Omega})$ satisfying the following conditions.
    \begin{enumerate}
        \item $0 \leq \varphi_{v,n} \leq 1$ and $\varphi|_{X \setminus D(v,2^{2-n})} = 0$.
        \item $\sum_{v \in V_n} \varphi_{v,n} = \mathds{1}_X$.
        \item $\mathcal{E}^{\Omega}(\varphi_{v,n}) \leq P \mu(D(v,2^{-n}))/\Psi(2^{-n})$.
    \end{enumerate}
    We note that ensuring (3) requires Lemma \ref{lemma:Leibniz}.
    Now, consider the functions
    \begin{equation*}
         f_n := \sum_{v \in Z_n} f_{D(v,2^{-n})} \cdot \varphi_{v,n}.
    \end{equation*}
    We first collect some of their properties.
    For each bounded open set $U \subseteq \widetilde{\Omega}$, $f_n|_U$ is finite linear combination of functions in $\{ \varphi_{v,n} \}_{v \in Z_n}$.
    Thus, $f \in \mathcal{F}_{\loc}(\Omega)$.
    Moreover, it follows from the metric doubling property of $\widetilde{\Omega}$ that each open ball $D(v,2^{-n})$ intersects at most $N = N(D)$ members of $\{ D(v,2^{-n}) \}_{v \in Z_n}$ where $D$ is the doubling constant of $\mu$. We then compute
    \begin{align*}
        & \quad \int_X f_n^2 \, d\mu = \int_X \left( \sum_{v \in Z_n} f_{D(v,2^{-n})} \cdot \varphi_{v,n} \right)^2 \, d\mu  \\
        \leq & \quad C(D) \sum_{v \in Z_n} \int_{D(v,2^{2-n})} (f_{D(v,2^{-n})})^2 \varphi_{v,n}^2 \, d\mu && (\text{Cauchy--Schwarz})\\
        \leq & \quad C(D) \sum_{v \in Z_n} \mu(D(v,2^{2-n})) (f_{D(v,2^{-n})})^2 \, d\mu && (0 \leq \varphi_{v,n} \leq 1)\\
        \leq & \quad C(D) \sum_{v \in Z_n} \frac{\mu(D(v,2^{2-n}))}{\mu(D(v,2^{-n}))}\int_{D(v,2^{2-n})} f^2 \, d\mu && (\text{Jensen's ineq.})\\
        \leq & \quad C'(D) \int_{X} f^2 \, d\mu. && (\text{Lemma \ref{lemma:IUD-doubling}})
    \end{align*}
    We use the argument in \cite[Proof of Theorem 3.12]{shimizu2025characterizations} to estimate the energy of $f_n$.
    Let $v \in Z_n$ and denote $I_n(v) := \{ w \in Z_n : D(v,2^{-n}) \cap D(w,2^{-n})  \neq \emptyset\}$.
    By the metric doubling property, we have the bound on the cardinality $\abs{I_n(v)} \leq A(D)$ depending only on the doubling constant of $\mu$.
    \begin{align*}
        & \quad \Gamma^\Omega\Span{f_n}(D(v,2^{-n}))\\
        = & \quad \Gamma^\Omega\Span{f_n - f_{D(v,2^{-n})} \mathds{1}_{X} }(D(v,2^{2-n})) && (\text{Strong locality}) \\
        = & \quad \Gamma^\Omega\left\langle \sum_{w \in I_n(v)} (f_{D(v,2^{-n})} - f_{D(w,2^{-n})}) \varphi_{v,n} \right\rangle (D(v,2^{2-n})) && (\text{Condition (2)}) \\
        \leq & \quad \sum_{ w \in I_n(v) }  (f_{D(v,2^{-n})} - f_{D(w,2^{-n})})^2 \sum_{w \in I_n(v)} \mathcal{E}^\Omega(\varphi_{w,n}) && (\text{Cauchy--Schwarz})\\
        \leq & \quad P \cdot A(D) \sum_{ w \in I_n(v) } (f_{D(v,2^{-n})} - f_{D(w,2^{-n})})^2 \frac{\mu(D(w,2^{-n}))}{\Psi(2^{-n})} && (\text{Condition (3)}) \\
        \leq & \quad K \sum_{ w \in I_n(v) } \frac{\Psi(2^{2-n})}{\mu(D(v,2^{-n}))}\Gamma^\Omega\Span{f}(D(w,\sigma 2^{-n})) \frac{\mu(D(w,2^{-n}))}{\Psi(2^{-n})} && (\text{Lemma \ref{lemma:PIRefl}})\\
        \leq & \quad K' \Gamma^\Omega\Span{f}(D(v,\sigma 2^{4-n})). && (\text{Lemma \ref{lemma:IUD-doubling}})
    \end{align*}
    By summing the previous inequality over all $v \in Z_n$, it follows from the metric doubling property that there is $K'' \geq 1$ independent of $n$,
    \[
        \mathcal{E}^\Omega(f_n) \leq K''\mathcal{E}^\Omega(f).
    \]
    
    Thus, we have verified that $\{f_n\}_{n = 1}^\infty \subseteq \mathcal{F}(\Omega) \cap C(X)$ is a bounded sequence in $(\mathcal{F}(\Omega),\mathcal{E}^\Omega_1)$.
    Moreover, $f_n \to f$ in $L^2(\widetilde{\Omega},\mu)$. To see this, note that $f_n \to f$ in $L^\infty(X,\mu)$ when $f \in C_c(\widetilde{\Omega})$ by the uniform continuity.
    The general case $f \in L^2(\widetilde{\Omega},\mu)$ follows from the density $C_c(\widetilde{\Omega}) \subseteq L^2(\widetilde{\Omega},\mu)$ and the fact that the linear operators $f \mapsto f_n$ are uniformly bounded in $L^2(\widetilde{\Omega},\mu)$.
    Hence, by the Mazur's lemma argument in Lemma \ref{lemma:FunAnal}, there are convex combinations
    \[
        \sum_{l = k}^{N_k} \lambda_{k,l} f_k \in \mathcal{F}(\Omega) \cap C(\widetilde{\Omega})
    \]
    that converge to $f$ as $k \to \infty$ in the Hilbert space $(\mathcal{F}(\Omega),\mathcal{E}_1^\Omega)$.
    This proves that $f$ can be approximated by continuous functions $\mathcal{F}(\Omega) \cap C(X)$.

    The remaining task is to show that $f \in \mathcal{F}(\Omega)$ can be approximated by compactly supported continuous functions. To this end, by the previous part of the proof, we may assume $f \in \mathcal{F}(\Omega) \cap C(X)$.
    We also fix $x \in \widetilde{\Omega}$ and for each $n \in \N$ we take a cutoff function $\xi_n \in \mathcal{F}(\Omega)$ for $D(x,2^n) \subseteq D(x,2^{n+1})$ provided by Corollary \ref{cor:CS}.
    Then, it follows from the reasoning in Remark \ref{rem:CS} that $f \cdot \xi_n \in \mathcal{F}(\Omega) \cap C_c(\widetilde{\Omega})$ and
    \[
        \mathcal{E}^\Omega(f \cdot \xi_n) \leq A \left(\int_{B(y,2^{n+1})} \, d\Gamma^\Omega\Span{f} + \frac{1}{\Psi(r)} \int_{B(y,2^{n+1})}f^2 \, d\mu\right),
    \]
    where $A \geq 1$ is independent of $n$.
    Hence, $f \cdot \xi_n$ is a bounded sequence in $(\mathcal{F}(\Omega),\mathcal{E}_1^\Omega)$ satisfying $f \cdot \xi_n \to f$ in $L^2(\widetilde{\Omega},\mu)$. By the identical Mazur's lemma argument, we conclude that $f$ can be approximated by functions in $\mathcal{F}(\Omega) \cap C_c(\widetilde{\Omega})$, which completes the proof.
\end{proof}

\subsection{Proof of the main theorem}

Now that the regularity of $(\mathcal{E}^\Omega,\mathcal{F}(\Omega))$ is verified, we may consider the associated energy measures. For now, we denote them $\widetilde{\Gamma}^\Omega\Span{\cdot}$ to distinguish them from $\Gamma^\Omega\Span{\cdot}$. Nevertheless, it follows fairly directly that these measures coincide; see \cite[Proof of Proposition 5.7]{kajino2023heat} for a similar argument.

\begin{corollary}\label{cor:EMcoincide}
    For all $f \in \mathcal{F}(\Omega)$ it holds that $\widetilde{\Gamma}^\Omega\Span{f} = \Gamma^\Omega\Span{f}$.
\end{corollary}

\begin{proof}
    During the proof, we let $\varphi \in \mathcal{F}(\Omega) \cap C_c(\widetilde{\Omega})$ to be arbitrary.
    
    First, let $f \in \mathcal{F}(\Omega) \cap C_c(\widetilde{\Omega})$. By Lemma \ref{lemma:Leibniz},
    \begin{align*}
        \mathcal{E}^\Omega(f,f \varphi) - \frac{1}{2} \mathcal{E}^\Omega(f^2,\varphi) & = \int_{\Omega} \varphi\, d\Gamma^\Omega\Span{f} + \int_{\Omega} f\, d\Gamma^\Omega\Span{f,\varphi} - \int_{\Omega} f\, d\Gamma^\Omega\Span{f,\varphi}\\
        & = \int_{\Omega} \varphi\, d\Gamma^\Omega\Span{f}.
    \end{align*}
    Since $\widetilde{\Gamma}^\Omega\Span{f}$ is uniquely determined by the formula in the previous display, we must have $\Gamma^\Omega\Span{f} =  \widetilde{\Gamma}^\Omega\Span{f}$.
    
    Then, we consider the case $f \in \mathcal{F}(\Omega) \cap L^\infty(\widetilde{\Omega},\mu)$, and take a sequence $\{ f_n \}_{n = 1}^\infty \subseteq \mathcal{F}(\Omega) \cap C_c(\widetilde{\Omega})$ convering to $f$ in the Hilbert space $(\mathcal{F}(\Omega),\mathcal{E}_1^\Omega)$.
    Since $\norm{f}_{L^\infty(X,\mu)} < \infty$, we may assume that $\sup_{n \in \N} \norm{f_n}_{L^\infty(X)} < \infty$. 
    Now, it follows from the Leibniz rule \cite[Theorem 3.2.5]{FOT},
    \begin{align*}
        \quad \frac{1}{2}& \abs{\mathcal{E}^\Omega(f_n^2,\varphi) - \mathcal{E}^\Omega(f^2,\varphi)} =  \left|\int_{\widetilde{\Omega}} f_n \, d\widetilde{\Gamma}^\Omega\Span{f_n,\varphi} - \int_{\widetilde{\Omega}} f \, d\widetilde{\Gamma}^\Omega\Span{f,\varphi} \right|\\
        \leq \quad  & \left|\int_{\widetilde{\Omega}} f_n \, d\widetilde{\Gamma}^\Omega\Span{f_n,\varphi} - \int_{\widetilde{\Omega}} f_n \, d\widetilde{\Gamma}^\Omega\Span{f,\varphi} \right| + \left|\int_{\widetilde{\Omega}} f_n \, d\widetilde{\Gamma}^\Omega\Span{f,\varphi} - \int_{\widetilde{\Omega}} f \, d\widetilde{\Gamma}^\Omega\Span{f,\varphi} \right|.
    \end{align*}
    We show that there is a subsequence such that the upper bound in the previous display vanishes as $n \to \infty$.
    
    By taking a subsequence if necessary, it follows from \cite[Theorem 2.1.4]{FOT} that $f_n$ converges point-wise to $f$ quasi-everywhere. Also note that energy measures do not charge sets of zero capacity by \cite[Lemma 3.2.4]{FOT}. Thus, by using the facts that $\abs{f_n},\abs{f} \leq M < \infty$ holds quasi-everywhere for some fixed $M$, and that $\widetilde{\Gamma}^\Omega\Span{f,\varphi}$ is a signed Radon measure with finite total variation, it follows from the dominated convergence theorem that the latter term in the previous display vanishes as $n \to \infty$.
    
    We estimate the first term as follows.
    Recall that the operator norm of a signed Radon measure $\nu$ as an element of the dual space of $(C_c(\widetilde{\Omega}),\norm{\cdot}_{L^\infty})$ is equal to its total variation $\abs{\nu}_{\text{TV}}$; see \cite[Theorem 6.19]{Rudin2}.
    Furthermore, it follows from the convergence $f_n \to f$ in $(\mathcal{F}(\Omega),\mathcal{E}_1^\Omega)$ that the sequence of Radon measures $\widetilde{\Gamma}^\Omega\Span{f_n,\varphi}$ converges to $\widetilde{\Gamma}^\Omega\Span{f,\varphi}$ in total variation. To see this, fix a Borel set $A \subseteq \widetilde{\Omega}$. Then, by the definition of the two variable measures,
    \begin{align*}
        & \abs{\widetilde{\Gamma}^\Omega\Span{f,\varphi}(A) - \widetilde{\Gamma}^\Omega\Span{f,\varphi}(A)}\\
        \leq \quad & \abs{\widetilde{\Gamma}^\Omega\Span{f+ \varphi}(A) - \widetilde{\Gamma}^\Omega\Span{f_n + \varphi}(A)} + \abs{\widetilde{\Gamma}^\Omega\Span{f- \varphi}(A) - \widetilde{\Gamma}^\Omega\Span{f_n - \varphi}(A)}.
    \end{align*}
    By using the reverse triangle inequality, which follows from the Cauchy--Schwarz inequality \cite[Lemma 5.6.1]{FOT},
    \begin{align*}
        & \abs{\widetilde{\Gamma}^\Omega\Span{f \pm \varphi}(A) - \widetilde{\Gamma}^\Omega\Span{f_n \pm \varphi}(A)}\\
        \leq & \quad\left( \sqrt{\mathcal{E}^\Omega(f\pm\varphi)} + \sqrt{\mathcal{E}^\Omega(f_n\pm\varphi)} \right) \sqrt{\mathcal{E}^\Omega(f - f_n)}\\
        \leq & \quad 2\left( \sqrt{\mathcal{E}^\Omega(f)} + \sqrt{\mathcal{E}^\Omega(f_n)} + \sqrt{\mathcal{E}(\varphi)} \right) \sqrt{\mathcal{E}^\Omega(f - f_n)}
    \end{align*}
    Since the sequence $\mathcal{E}^\Omega(f_n)$ is bounded, we see from the previous two displays that $\widetilde{\Gamma}^\Omega\Span{f_n,\varphi}$ converges to $\widetilde{\Gamma}^\Omega\Span{f,\varphi}$ in total variation.
    We conclude that
    \[
        \left|\int_{\widetilde{\Omega}} f_n \, d\widetilde{\Gamma}^\Omega\Span{f_n,\varphi} - \int_{\widetilde{\Omega}} f_n \, d\widetilde{\Gamma}^\Omega\Span{f,\varphi} \right| \leq \sup_{n \in \N} \norm{f_n}_{L^\infty} \abs{\widetilde{\Gamma}^\Omega\Span{f_n,\varphi} - \widetilde{\Gamma}^\Omega\Span{f,\varphi}}_{\text{TV}}
    \]
    vanishes as $n \to \infty$, which gives $\mathcal{E}^\Omega(f_n^2,\varphi) \to \mathcal{E}^\Omega(f^2,\varphi)$ as $n \to \infty$.

    By a similar argument, $\mathcal{E}^\Omega(f_n,f_n\varphi) \to \mathcal{E}^\Omega(f,f\varphi)$ as $n \to \infty$. Thus, we have
    \begin{align*}
        \int_{\Omega}\varphi \, d\Gamma^{\Omega}\Span{f} & = \lim_{n \to \infty} \int_{\Omega} \varphi \, d\Gamma^{\Omega}\Span{f_n}\\
        & = \lim_{n \to \infty} \mathcal{E}^\Omega(f_n,f_n \varphi) - \frac{1}{2} \mathcal{E}^\Omega(f_n^2,\varphi)\\
        & = \mathcal{E}^\Omega(f,f \varphi) - \frac{1}{2} \mathcal{E}^\Omega(f^2,\varphi).
    \end{align*}
    The first equality follows from the fact that $\Gamma^\Omega\Span{f_n}$ converges to $\Gamma^\Omega\Span{f}$ in total variation by a similar argument as above.
    This completes the proof in the case $f \in \mathcal{F}(\Omega) \cap L^\infty(\widetilde{\Omega},\mu)$.

    Lastly, we consider the general case $f \in \mathcal{F}(\Omega)$. Note that we have defined the energy measure $\widetilde{\Gamma}^\Omega$ in Definition \ref{def:Dir} according to
    \begin{equation}\label{eq:Final}
        \widetilde{\Gamma}^\Omega\Span{f}(A) = \lim_{k \to \infty} \widetilde{\Gamma}^\Omega\Span{(f \land k) \lor -k}(A) = \lim_{k \to \infty}\Gamma^\Omega\Span{(f \land k) \lor -k}(A)
    \end{equation}
    for all Borel sets $A \subseteq \widetilde{\Omega}$.
    The last equality in \eqref{eq:Final} follows from the previous part of the proof. Furthermore, we have the following properties.
    \begin{enumerate}
        \vspace{2pt} 
        \item $(f \land k) \lor -k \to f$ in $L^2(\widetilde{\Omega},\mu)$ by the dominated convergence theorem.
        \vspace{2pt}
        \item $\mathcal{E}^\Omega((f \land k) \lor -k) \leq \mathcal{E}^\Omega(f)$ by the Markov property.
    \end{enumerate}
    Thus, according to Lemma \ref{lemma:FunAnal},
    \[
    (f \land k) \lor -k \to f \text{ in } (\mathcal{F}(\Omega),\mathcal{E}_1^\Omega).
    \]
    By a similar argument as earlier in the proof,
    \[
        \Gamma^\Omega\Span{ (f \land k) \lor -k} \to \Gamma^\Omega\Span{f} \text{ in total variation.}
    \]
    This along with \eqref{eq:Final} proves $\widetilde{\Gamma}^\Omega\Span{f} = \Gamma^\Omega\Span{f}$ in the desired general case.
\end{proof}

\begin{remark}
    It follows from Corollary \ref{cor:EMcoincide} and the definition of the measures $\Gamma^\Omega\Span{\cdot}$ that the energy measures of $(\mathcal{E}^\Omega,\mathcal{F}(\Omega))$ do not charge the set $\widetilde{\Omega}\setminus\Omega$.
\end{remark}

We have now gathered everything we need to prove the main theorem, Theorem \ref{thm:Main}, and more generally, Theorem \ref{thm:MainMain}.

\begin{proof}[Proof of Theorem \ref{thm:MainMain}]
    First, $(\mathcal{E}^\Omega,\mathcal{F}(\Omega))$ is a strongly local regular Dirichlet form on $L^2(\widetilde{\Omega},\mu)$ by Proposition \ref{prop:SLRDF}. By Corollary \ref{cor:EMcoincide}, the associated energy measures are the measures $\Gamma^\Omega\Span{\cdot}$ given in Definition \ref{def:LocDir}. Hence, $(\mathcal{E}^\Omega,\mathcal{F}(\Omega))$ satisfies the Poincar\'e inequality \ref{eq:PI} and the cutoff enery condition \ref{eq:CE} by Lemma \ref{lemma:PIRefl} and Lemma \ref{lemma:cutoffs}, respectively. Hence, \ref{eq:HKMain} follows from Lemma \ref{lemma:charHK}.
\end{proof}

\bibliographystyle{acm}
\bibliography{DF}

\end{document}